\documentclass[11pt]{article}
\usepackage{amsmath,amssymb,amsthm,comment}
\usepackage{amsfonts}
\usepackage{color}
\usepackage{epsfig}
\newtheorem{theo}{Theorem}[section]
\newtheorem{lemma}[theo]{Lemma}

\newtheorem{prop}[theo]{Proposition}

\makeatletter
    
    \@addtoreset{equation}{section}
  \makeatother
\newcommand{\io}{\int_0^1}
\newcommand{\R}{\mathbb{R}}

\newcommand{\dx}{\partial_x}
\newcommand{\dt}{\partial_t}
\newcommand{\abs}{\\[5pt]}
\newcommand{\tm}{T_{\mathrm{max}}}
\begin{document}
\title{No critical nonlinear diffusion in 1D quasilinear fully parabolic chemotaxis system}
\author{
Tomasz Cie\'slak\footnote{cieslak@impan.pl}\\
{\small Institute of Mathematics,}\\
{\small Polish Academy of Sciences,}\\
{\small Warsaw, 00-656, Poland}
\and
Kentarou Fujie\footnote{fujie@rs.tus.ac.jp}\\
{\small Department of Mathematics,}\\
{\small Tokyo University of Science,}\\
{\small Tokyo, 162-0861, Japan}
}
\date{\small\today}
\maketitle
\begin{abstract}
This paper deals with the fully parabolic 1d chemotaxis system with diffusion $1/(1+u)$. 
We prove that the above mentioned nonlinearity, despite being a natural candidate, is not critical. 
It means that for such a diffusion any initial condition, independently on the magnitude of mass, generates global-in-time solution. In view of our theorem one sees that one-dimensional Keller-Segel system is essentially 
different than its higher-dimensional versions.  In order to prove our theorem  we establish a new Lyapunov-like functional associated to the system. The information we gain from our new functional (together with some estimates based on the well-known old Lyapunov functional) turns out to be rich enough to establish global existence for the initial-boundary value problem.
\abs
\end{abstract}
\newpage
\section{Introduction}\label{section_introduction}
%
%
%

We consider the one-dimensional version of the following quasilinear Keller-Segel problem 
\begin{align}\label{P1}
	\begin{cases}
	\dt u =\dx \left( a(u) \dx u -  u \dx v \right)
	&\mathrm{in}\ (0,T)\times(0,1), \\[1mm]
	 \dt v =\dx^2 v -v+u
	&\mathrm{in}\ (0,T)\times(0,1),  \\[1mm]
	 \dx A(u)(t,0)=\dx A(u)(t,1)=\dx v(t,0)=\dx v(t,1)=0
	&t\in(0,T),  \\[1mm]
	u(0,\cdot)=u_0,\quad v(0,\cdot)=v_0,
	&\mathrm{in}\ (0,1),
	\end{cases}
\end{align} 
where $a$ is a positive function $a\in C^1(0,\infty)\cap C[0,\infty)$. 
The function $A$ is an indefinite integral of $a$ and 
$u_0 \in C^1[0,1]$ such that 
$u_0 \geq 0$ in $(0,1)$.
Furthermore we assume $0 \leq v_0 \in C^1[0,1]$.

The particular choice of nonlinear diffusion $a(u)=1/(1+u)$ is important since
such a diffusion is a candidate for a critical one in one-dimensional setting.
Namely in dimensions $n\geq 2$ $a(u)=(1+u)^{1-\frac{2}{n}}$ is critical in the 
sense that it distinguishes between the global-in-time existence for any initial 
data for stronger diffusions (see \cite{ss}) and finite-time blowups when the 
diffusion is weaker, see \cite{CS}.
Next, in the particular case of diffusion given by $a(u)=(1+u)^{1-\frac{2}{n}}$,
solutions exist for small mass data while they blow up in finite time for initial 
masses large enough, see \cite{miz_win} in dimension $2$ and \cite{lau_miz} in dimensions
$3$ and $4$. An interested reader might find more details in \cite{BBTW}. 

In dimension $1$ a situation is similar, when diffusion is subcritical, namely 
$a(u)=(1+u)^{-p}, p<1$ solutions emanating from any data (regardless the magnitude
of mass) exist globally in time and are bounded, see \cite{BCM-R}, while in the supercritical 
case ($p>1$) solutions blowing up in finite time (only for big masses and under some additional 
restrictions) have been constructed in \cite{TCPhLAIHP}. In the critical case 
$a(u)=1/(1+u)$ so far only existence of global solutions for initial masses small enough 
are known (\cite{BCM-R}). Our aim is to study this case fully. Before we introduce our main result 
let us mention that similar results are known also in the parabolic-elliptic version 
of quasilinear Keller-Segel system. The  higher dimensional problem in bounded domain is 
treated in \cite{CW, Nasri}
(global existence in the subcritical and critical cases for small masses, respectively) and 
\cite{CW} (finite-time blowups in supercritical case for any initial mass), \cite{CL_CRAS}
(in the critical case for mass large enough). The same problem in the whole space
has been solved in \cite{BCL, S_DIE, S_ADE}. One-dimensional case has been solved in
\cite{CL_DCDS}, where a peculiar change of variables was used by the authors. As a consequence
both blowup in the supercritical case and global existence in the subcritical one have been obtained.
Surprisingly, also global existence in the critical case $a(u)=1/(1+u)$ was obtained. However, all
the reasoning depends on the crucial change of variables. The change of variables works only 
in the parabolic-elliptic case (moreover, it is also very sensitive to the fact that the J\"{a}ger-Luckhaus
type simplification is studied, the usual Keller-Segel type system was carried in the recent 
note \cite{CF1}). The fully parabolic case was an open problem for several years. We answer this
case in the present paper. Let us notice that also in the case of nonlocal diffusions in 1d,
at least in the parabolic-elliptic case, critical diffusion does not exist, see \cite{bg}.       

To this end, we construct the following new functional 
associated to \eqref{P1} (it is worth noticing that it holds only in dimension $1$) 
satisfying
\begin{eqnarray*}
\dfrac{d}{dt} \mathcal{F}(u(t)) + \mathcal{D}(u(t),v(t))
=\io \dfrac{ua(u)(v+\dt v)^2}{4},
\end{eqnarray*}
where
\begin{eqnarray}\label{eF}
\mathcal{F}(u(t)) 
&:=&\frac{1}{2} \io \dfrac{(a(u))^2}{u}|\dx u|^2
- \io u\int_1^u a(r)\,dr ,\\[2mm]
\mathcal{D}(u(t),v(t)) 
&:=& \io ua(u) 
\left|  \dx \left(\dfrac{a(u)}{u}\dx u \right) - \dx^2 v  + \dfrac{(v+\dt v)}{2} \right|^2.
\label{De}
\end{eqnarray}
Thanks to the known facts concerning the usual Lyapunov functional related to \eqref{P1} 
(that will be introduced later as $L$)
we notice that the growth of ${\cal F}$ along the trajectories can be controlled. Then we obtain the following main result 
answering the open question concerning global existence in the critical quasilinear fully parabolic 
1d Keller-Segel.  
\begin{theo}\label{main_theorem}
Let  $a(u)=\frac{1}{1+u}$ and both $u_0, v_0\geq 0$. 
Then the problem \eqref{P1} has a unique classical positive solution, which exists globally in time. 
\end{theo}
%
%
%
%
%
%
%
\section{Preliminaries}\label{section_preliminaries}
The next lemma contains a crucial identity. It was shown in \cite[Lemma 2.1]{CF1}
of the accompanying paper. As noticed in \cite[Remark 2.2]{CF1}, the equality below 
holds only in dimension $1$.
\begin{lemma}\label{Key}
Let $\phi \in C^3(0,1)$. 
Then the following identity holds:
\begin{align*}
\phi \dx \mathcal{M}(\phi)
=\dx \left( \phi a(\phi) \dx \left(\dfrac{a(\phi)}{\phi}\dx \phi \right) \right),
\end{align*}
where 
\begin{align*}
\mathcal{M}(\phi) 
:= \dfrac{a(\phi)a'(\phi)}{\phi} |\dx \phi|^2
-\dfrac{(a(\phi))^2}{2\phi^2} |\dx \phi|^2
+\dfrac{(a(\phi))^2}{\phi}\dx^2 \phi.
\end{align*}
\end{lemma}
Next, we have several well-known facts.
The following inequality is obtained in \cite{BHN, NSY, BCM-R}.
\begin{lemma}\label{BHN_inequality}
For $w\in H^1(0,1)$ and any $\delta>0$ there exists $C_\delta>0$ 
such that
\begin{align*}
\|w\|^4_{L^4(0,1)}
 \leq \delta \|w\|^2_{H^1(0,1)} \io |w \log w| 
 + C_\delta \|w\|_{L^1(0,1)}.  
\end{align*}
\end{lemma}
The following local existence is known, see \cite{TCPhLAIHP, BCM-R}.
\begin{lemma}\label{prop_local_existence}
 For $a\in C^1(0,\infty)\cap C[0,\infty)$ and nonnegative $(u_0, v_0)\in L^{\infty}(0,1)\times W^{1,\infty}(0,1)$ there exist $T_{\mathrm{max}}\leq \infty$
 {\rm(}depending only on ${\|{u_0}\|}_{L^{\infty}}$ and ${\|{v_0}\|}_{W^{1,\infty}}${\rm)} 
 and exactly one pair $(u,v)$ of positive functions
 \begin{align*}
(u,v) \in C([0,\tm)\times [0,1]; \R^2) \cap
C^{1,2}((0,\tm)\times [0,1]; \R^2)
   \end{align*}
 that solves {\rm(\ref{P1})} in the classical sense. 
 Also, the solution $(u,v)$ satisfies the mass identities
   \begin{equation*}
     \int_{\Omega}u(x,t)\,dx=\int_{\Omega}u_0(x)\,dx
     \quad \text{for\ all}\ t \in (0,T_{\mathrm{max}}).
   \end{equation*}
In addition, if $T_{\mathrm{max}} < \infty$, then
   \begin{equation*}
     \limsup_{t \nearrow T_{\mathrm{max}}}
     \left( {\|{u(t)}\|}_{L^{\infty}(0,1)}
     +{\|{v_0}\|}_{W^{1,\infty}(0,1)} \right) 
     = \infty.
   \end{equation*}   
\end{lemma}
In virtue of the conservation of the total mass $\|u\|_{L^1(\Omega)}$, we can get the following regularity estimates by the semigroup estimates.
\begin{lemma}\label{prop_ellipticReg}
There exists some constant $M=M(\|u_0\|_{L^1(0,1)},p, \left\|v_0\right\|_{L^p(0,1)})>0$ such that
\begin{align*}
\sup_{t\in[0,\tm)}\|v\|_{L^{p}(0,1)} \leq M,
\end{align*}
where $p \in [1,\infty)$.
\end{lemma}

Finally, let us recall some facts concerning the well-known Lyapunov functional.
In the presentation we refer to \cite[Lemma 4, Lemma 5]{TCPhLAIHP}.
The following functional
$L(u,v):=\io b(u)-\io uv +1/2\left\|v\right\|^2_{H^1(0,1)}$
satisfies
\[
\frac{d}{dt}L(u(t),v(t))=-\io v_t^2-\io u|(b'(u)-v)_x|^2,
\]
where $b\in C^2(0,\infty)$ is such that $b''(r)=\frac{a(r)}{r}$ for $r>0$ and $b(1)=b'(1)=0$.

Moreover, $L$ is bounded from below and so in particular there exists $C>0$
such that for any $t<T_{max}$ 
\begin{equation}\label{wazne}
\int_0^t\int_0^1 \left( v_t(x,s) \right)^2dxds\leq C. 
\end{equation}
%
%
%
%
%
%
%
%
%
%
%
\section{New Lyapunov-like functional}\label{section_new_lyapunov}
In this section we construct a functional 
associated to the problem \eqref{P1}. It does not decrease along the trajectories, so it is not a classical Lyapunov functional. However, along the trajectories, we control its growth thanks to the information coming delivered by the Lyapunov functional $L$. Then we are able to derive the required estimates.

We proceed in several steps. 
Basing on Lemma \ref{Key}, since in \cite[Lemma 3.1]{CF1} we make use only of the first equation
in \eqref{P1}, we have exactly the same way as in \cite[Lemma 3.1]{CF1}. 
\begin{lemma}\label{lemma_Lyap1}
Let $(u,v)$ be a solution of \eqref{P1} in $(0,T)\times (0,1)$. 
Then the following identity holds
\begin{align*}
\dfrac{d}{dt} \left(\frac{1}{2} \io \dfrac{(a(u))^2}{u}|\dx u|^2 \right)
+ \io u a(u) \left| \dx \left(\dfrac{a(u)}{u}\dx u \right) \right|^2
= \io u a(u) \dx^2 v \cdot \dx \left(\dfrac{a(u)}{u}\dx u \right).
\end{align*}
\end{lemma}
Next, owing to the following two straightforward equalities
\begin{align*}
&-\io a(u) \dx u \cdot a(u) \dx u - \io \dx (ua(u)) a(u) \dx u
= \io \dx \left(\dfrac{a(u)}{u} \dx u \right) u^2a(u) ,\\
&\io u \dx v \cdot a(u) \dx u + \io \dx (ua(u)) u \dx v 
 = \io \dx \left( u^2 a(u) \right) \cdot \dx v,
\end{align*}
testing the first equation of \eqref{P1} by $\int_1^u a(r)\,dr + u a(u)$ 
and integrating over $(0,1)$, we arrive at
\begin{lemma}\label{Est1_P1}
Let $(u,v)$ be a solution of \eqref{P1} in $(0,T)\times (0,1)$. 
Then the following identity holds
\begin{align*}
\dfrac{d}{dt} \left(
\io u\int_1^u a(r)\,dr 
 \right)
=\io \dx \left(\dfrac{a(u)}{u} \dx u \right) u^2a(u)
+\io \dx \left( u^2 a(u) \right) \cdot \dx v
\end{align*}
\end{lemma}
Finally, we introduce a crucial observation. For ${\cal F}$ and ${\cal D}$ given 
in \eqref{eF} and \eqref{De} respectively, the following formula holds. 
\begin{lemma}\label{prop_new_lyapunov}
Let $(u,v)$ be a solution of \eqref{P1} in $(0,T)\times (0,1)$. 
The following identity is satisfied
\begin{eqnarray*}
\dfrac{d}{dt} \mathcal{F}(u(t)) + \mathcal{D}(u(t),v(t))
=\io \dfrac{ua(u)(v+\dt v)^2}{4}.
\end{eqnarray*}
\end{lemma}
\begin{proof}
Multiplying \eqref{P1} by $u a(u) \dx^2 v$ and integrating over $(0,1)$ we have that
\begin{align}\label{Est2_P1}
\notag 
\io ua(u) \dt v \dx^2 v&=\io u a(u)  |\dx^2 v|^2  - \io  v \cdot u a(u) \dx^2 v  
+ \io u^2a(u) \dx^2 v\\
  &=  \io u a(u)  |\dx^2 v|^2  - \io  v \cdot u a(u) \dx^2 v  
  - \io \dx \left( u^2 a(u) \right) \cdot \dx v.
\end{align} 
Combining Lemma \ref{Est1_P1} and \eqref{Est2_P1} we get
\begin{align*}
&\,\frac{d}{dt} \left(
- \io u\int_1^u a(r)\,dr 
 \right)
+\io ua(u) |\dx^2 v|^2  - \io  v \cdot ua(u) \dx^2 v
-\io ua(u) \dt v \dx^2 v\\
&=- \io \dx \left( \frac{a(u)}{u} \dx u \right) \cdot u^2 a(u),
\end{align*}
and then using the second equation of \eqref{P1} we see that
\begin{align*}
&\frac{d}{dt} \left(
- \io u\int_1^u a(r)\,dr 
 \right)
+\io ua(u) |\dx^2 v|^2  - \io  v \cdot ua(u) \dx^2 v -\io ua(u) \dt v \dx^2 v\\
&=- \io ua(u) \dx \left( \frac{a(u)}{u} \dx u \right)(\dt v - \dx^2 v + v).
\end{align*}
Thus it follows from Lemma \ref{lemma_Lyap1} that
\begin{align*}
&\dfrac{d}{dt} \left(
\frac{1}{2} \io \dfrac{(a(u))^2}{u}|\dx u|^2
- \io u\int_1^u a(r)\,dr
\right)
+ \io ua(u)
\left| \dx \left(\dfrac{a(u)}{u}\dx u \right) - \dx^2 v \right|^2\\
& + \io ua(u) (v+\dt v) \cdot \left( \dx \left(\dfrac{a(u)}{u}\dx u \right) - \dx^2 v \right) =0,
\end{align*}
that is,
\begin{align*}
&\dfrac{d}{dt} \left(
\frac{1}{2} \io \dfrac{(a(u))^2}{u}|\dx u|^2
- \io u\int_1^u a(r)\,dr
\right)
+ \io ua(u) 
\left|  \dx \left(\dfrac{a(u)}{u}\dx u \right) - \dx^2 v  + \dfrac{v+\dt v}{2} \right|^2\\
&= \io \dfrac{ua(u)(v+\dt v)^2}{4}.
\end{align*}
which is the desired inequality.
\end{proof}
From now on, we focus on the critical case $a(u)=1/(1+u)$.
Then ${\cal F}$ defined in \eqref{eF} takes the form
\begin{equation}\label{eFa}
{\cal F}=\dfrac{1}{2} \io \dfrac{|\dx u|^2}{u(1+u)^2}-\io u\log (1+u).
\end{equation}
As mentioned before we can control the growth of ${\cal F}$.
\begin{prop}\label{lemma_classical_Lyap} 
Let $(u,v)$ be a solution of \eqref{P1} in $(0,T)\times (0,1)$ with nonlinear diffusion $a(u)=1/(1+u)$. Then
there exists a constant $C>0$ such that for any $t<T_{max}$
\begin{align*}
\int_0^t\int_0^1 \frac{ua(u)(v+v_t)^2}{4}\leq C(t+1).
\end{align*}
\end{prop}
\begin{proof}
In view of the choice of $a$ $ua(u)\leq 1$, moreover by Lemma \ref{prop_ellipticReg} and \eqref{wazne}
both $v$ and $v_t$ belong to $L^2(0,t;L^2(0,1))$ for any $t<T_{max}$.
\end{proof}
%
%
%
%
%
%
%
%
%
%
%
%
%
%
%
\section{Proof of Main theorem}\label{section_proof}
By Lemma \ref{prop_new_lyapunov} and Proposition \ref{lemma_classical_Lyap} we obtain the upper bound of the functional $\mathcal{F}(u)$. 
To derive the required estimates, we first establish the lower bound of the functional. 
\begin{prop}\label{propo}
Let $(u,v)$ be a solution of \eqref{P1} with $a(u)=1/(1+u)$ 
in $(0,T)\times (0,1)$ and $u_0\geq 0$, $\io u_0 = M>0$. 
The following estimates hold
\begin{align}\label{RegEst1}
\mathcal{F}(u) \geq 
\dfrac{1}{4} \io \dfrac{|\dx u|^2}{u(1+u)^2} - C(M)
\end{align}
with some $C=C(M)>0$, and
\begin{align}\label{RegEst2.5}
\io u\log (1+u)
\leq 
C(M) + 1/4 \left( \io \dfrac{|\dx u|^2}{u(1+u)^2} \right)^{\frac{1}{2}}.
\end{align}
\end{prop}
\begin{proof}
Let $t \in (0,T)$ be fixed. 
Since $\io u(t,x)\,dx = M$, we can find some point 
$x_0 \in [0,1]$ such that $u(t, x_0) = M$. Then
\begin{align*}
- \io u(x) \log (1+u(x))\,dx  &= \io u(x) \log \dfrac{1}{1+u(x)}\,dx\\
&= \io u(x) \left(\log \dfrac{1}{1+u(x)} - \log \dfrac{1}{1+u(x_0)} \right)\,dx + \io u(x)\log \dfrac{1}{1+u(x_0)}\,dx\\
&= \io u(x) \left(\int_{x_0}^x \dx \left( \log \dfrac{1}{1+u(z)} \right) \,dz \right)\,dx
- M \log (1+M).
\end{align*}
By the Cauchy--Schwarz inequality it follows that
\begin{align}\label{RegEst3}
\notag
- \io u \log (1+u) & \geq
-  \io u(x) \left(\int_{x_0}^x \frac{1}{u(z)} \cdot \dfrac{|\dx u(z)|^2}{(1+u(z))^2} \,dz \right)^{\frac{1}{2}}
\left(\int_{x_0}^x u(z) \,dz \right)^{\frac{1}{2}}\,dx
- M \log (1+M)\\
& \geq - M^{\frac{3}{2}} \left( \io \dfrac{|\dx u|^2}{u(1+u)^2} \right)^{\frac{1}{2}} 
- M \log (1+M).
\end{align}
From the above inequality we derive the lower bound for $\mathcal{F}(u)$ such that
\begin{align*}
\mathcal{F}(u) &\geq
\frac{1}{4} \io \dfrac{|\dx u|^2}{u(1+u)^2} + \frac{1}{4} \io \dfrac{|\dx u|^2}{u(1+u)^2} 
- M^{\frac{3}{2}} \left( \io \dfrac{|\dx u|^2}{u^3} \right)^{\frac{1}{2}} 
+ M^3 - M^3 - M \log (1+M)\\
& = \frac{1}{4} \io \dfrac{|\dx u|^2}{u(1+u)^2}
+ \dfrac{1}{4} \left(\left( \io \dfrac{|\dx u|^2}{u(1+u)^2} \right)^{\frac{1}{2}} - 2M^{\frac{3}{2}} \right)^2 
- M^3 - M\log (1+M),
\end{align*}
which gives us in turn \eqref{RegEst1}. Next, owing to the form of ${\cal F}$ in \eqref{eFa} and 
\eqref{RegEst1} we immediately see that
\begin{align*}
\io u\log(1+u)\leq \left(\frac{1}{2}-\frac{1}{4}\right) 
\left(\io \dfrac{|\dx u|^2}{u(1+u)^2}\right)^{\frac{1}{2}} + C(M),
\end{align*}
which gives us \eqref{RegEst2.5}. 
\end{proof}
Below we obtain regularity estimates which depend on the time interval $T>0$.
\begin{prop}\label{prop_regest}
Let $(u,v)$ be a solution of \eqref{P1} with $a(u)=1/(1+u)$ 
in $ (0,T) \times (0,1)$. Then  there exists some constant $C>0$ such that
\begin{align*}
\io u \log (1+u) + \io \frac{|\dx u|^2}{u(1+u)^2} \leq C(1+T)
\qquad\mbox{for all }t \in (0,T).
\end{align*}
\end{prop}
\begin{proof}
Due to Proposition \ref{lemma_classical_Lyap} we have the existence of constant $C>0$ such that
\begin{align*}
\int_0^t\io \dfrac{ua(u)(v+v_t)^2}{4} \leq C(t+1),
\end{align*}
it follows that for all $t \in (0,T)$,
\begin{align*}
\mathcal{F}(u(t)) \leq \mathcal{F}(u_0) + C(T+1).
\end{align*}
Thus \eqref{RegEst1} implies that
\begin{align*}
\dfrac{1}{4} \io \dfrac{|\dx u|^2}{u(1+u)^2} 
\leq \mathcal{F}(u(t)) + C(M)
\leq \mathcal{F}(u_0) + C(M)+ C(T+1)
\end{align*}
with some $C(M)>0$. Next, \eqref{RegEst2.5}
gives the claim. 
\end{proof}
\begin{proof}[Proof of Theorem \ref{main_theorem}]
With the information in Proposition \ref{prop_regest} as well as inequality in Lemma \ref{BHN_inequality} we can use \cite[Lemma 3]{BCM-R} to deduce that there exists some constant $C(T)>0$ such that 
\begin{align*}
\io (1+u)^3 \leq C(T+1)
\qquad\mbox{for all }t \in (0,T).
\end{align*}
By the iterative argument (see \cite[Proof of Theorem 1]{BCM-R} or \cite{NSY}) we have for any $p \in (1,\infty)$
\begin{align*}
\|1+u(t)\|_{L^p(0,1)} \leq C(T+1) \qquad \mbox{for all }t\in(0,T)
\end{align*}
with some $C(T)>0$. Finally by the standard regularity estimates for quasilinear parabolic equation (\cite[Proposition 3]{BCM-R}) we can derive boundedness of $u$,
\begin{align*}
\|1+u(t)\|_{L^\infty(0,1)} \leq C(T+1)  \qquad \mbox{for all }t\in(0,T),
\end{align*}
which implies global existence of solutions to \eqref{P1}. 
\end{proof}
%
%
%
%

\bigskip
\bigskip
\textbf{Acknowledgments} \\
The second author wishes to thank Institute of Mathematics of the Polish Academy of Sciences and WCMS, where he held a post-doctoral position, for financial support and the warm hospitality.
\end{document}